\documentclass{amsart}
\usepackage{amssymb, amsmath, amsthm, graphics, comment, xspace, enumerate}
\newtheorem{thm}{Theorem}[section]

\theoremstyle{definition}
\newtheorem{defn}[thm]{Definition}
\newtheorem{example}[thm]{Example}
\theoremstyle{remark}
\newtheorem{rem}[thm]{Remark}
\numberwithin{equation}{section}

\newcommand{\NN}{\mathbb N}

\newcommand{\DD}{\mathcal D}

\begin{document}
\title[(Q-)(C-EDS) and (q-)(C-EUDS) in locally convex spaces]{Quasi-equicontinous exponential families of generalized function $C$-semigroups in locally convex spaces}
\author{Marko Kosti\' c}
\address{Faculty of Technical Sciences,
University of Novi Sad,
Trg D. Obradovi\' ca 6, 21125 Novi Sad, Serbia}
\email{marco.s@verat.net}

\author{Stevan Pilipovi\' c}
\address{Department for Mathematics and Informatics,
University of Novi Sad,
Trg D. Obradovi\' ca 4, 21000 Novi Sad, Serbia}
\email{pilipovic@dmi.uns.ac.rs}

\author{Daniel Velinov}
\address{Department for Mathematics, Faculty of Civil Engineering, Ss. Cyril and Methodius University, Skopje,
Partizanski Odredi
24, P.O. box 560, 1000 Skopje, Macedonia}
\email{velinovd@gf.ukim.edu.mk}

{\renewcommand{\thefootnote}{} \footnote{
\\ \text{  }  \ \    {\it Key words and phrases.} $C$-distribution semigroups, $C$-ultradistribution semigroups,
integrated $C$-semigroups, convoluted $C$-semigroups, well-posedness, locally convex spaces.
\\  \text{  }  \ \ This research is partially supported by grant 174024 of Ministry
of Science and Technological Development, Republic of Serbia.}}

\begin{abstract}
Our main goal in this paper is to investigate the (q-)exponential $C$-distribution semigroups and (q-)exponential $C$-ultradistribution semigroups in the setting of sequentially complete locally convex spaces.
We contribute to previous work and the work of many other authors, providing additionally plenty of various examples and applications of obtained results.
\end{abstract}
\maketitle

\section{Introduction and Preliminaries}
Theory of non-densely defined infinitesimal generators  for ultradistribution, hyperfunction semigroups, $C$-regularized generalized semigroups can be found in \cite{fat1}, \cite{knjigah}, \cite{ko98}, \cite{diff1}, \cite{ptica}, \cite{tica}, \cite{ku112}, \cite{lar}  \cite{me152}, \cite{mija}, \cite{w241}, while for various subclasses of $K$-convoluted $C$-semigroups in Banach and locally convex spaces we refer to \cite{a43}, \cite{b41}, \cite{b42}, \cite{c51}, \cite{d81}, \cite{l1}, \cite{hiber1}, \cite{klm1}, \cite{knjigah}-\cite{segedin}, \cite{kuo1}, \cite{l114}, \cite{me152} and \cite{x263}; concerning
the theory of locally equicontinuous $C_{0}$-semigroups in locally convex spaces, special attention should be made of
\cite{choe}, \cite{dembart}, \cite{komatsulocally}, \cite{komura}, \cite{moore}
(see the monographs \cite{a43}, \cite{engel}, \cite{fat1}, \cite{hf} and \cite{pa11} for the Banach space case). The study of distribution semigroups in locally convex spaces has been initiated by R. Shiraishi, Y. Hirata \cite{1964}, T. Ushijima \cite{ush1}. For the best knowledge of the authors, there is no reference which treats ultradistribution semigroups in locally convex spaces, except \cite{cdscuds}. In this paper we continue to study such semigroups by the investigation of a class of quasi-equicontinuous exponential (q-exponential) $C$-distribution and ultradistribution semigroups, see Section 2.

We furnish several illustrative examples of application of our theoretical results in Fr\' echet spaces in Section 3. Actually, Section 3 gives motivation for this abstract setting of the theory.

The notation in this paper is standard.
By $E$ we denote a Hausdorff sequentially complete
locally convex space over the field of complex numbers, SCLCS for short.
The fundamental system of seminorms which defines the topology of $E$ is denoted by $\circledast_{E}$.
By $L_{\circledast}(E)$ we denote the subspace of $L(E)$ containing all continuous linear mappings $T$ from $E$ into $E$ such that for each $p\in \circledast$ there exists $c_{p}>0$ such that
$p(Tx)\leq c_{p} p(x),$ $x\in E.$
Denote by ${\mathcal B}$ the family of bounded subsets of $E,$ and
let $p_{B}(T):=\sup_{x\in B}p(Tx),$ $p\in \circledast_{X},$ $B\in
{\mathcal B},$ $T\in L(E,X).$ Then $p_{B}(\cdot)$ is a seminorm\index{seminorm} on
$L(E,X)$ and the system $(p_{B})_{(p,B)\in \circledast_{X} \times
{\mathcal B}}$ induces the Hausdorff locally convex topology on
$L(E,X).$
Set
$p_{A}(x):=p(x)+p(Ax),$ $x\in D(A),$ $p\in \circledast$.
Then the calibration $(p_{A})_{p\in \circledast}$ induces the Hausdorff sequentially complete locally convex topology on $D(A);$
we denote this space simply by $[D(A)].$\\
Define
$\Sigma_{\alpha}:=\{ z\in {\mathbb C} \setminus \{0\} :
|\arg (z)|<\alpha \}$ ($\alpha \in (0,\pi]$). If $f:\mathbb{R}\to\mathbb{C}$, put $(\tau_tf)(s):=f(s-t)$, $s\in\mathbb{R}$, $t\in\mathbb{R}$. The exponential region $E(a,b)$ has been defined for the first time by W. Arendt, O. El--Mennaoui and V. Keyantuo \cite{a22} as follows:
$$
E(a,b):=\Bigl\{\lambda\in\mathbb{C}:\Re\lambda\geq b,\:|\Im\lambda|\leq e^{a\Re\lambda}\Bigr\} \ \ (a,\ b>0).
$$
If $X$ is a general topological vector space,
then a function $f :
\Omega \rightarrow X,$ where $\Omega$ is an open subset of ${\mathbb
C},$ is said to be analytic if it is locally expressible in a
neighborhood of any point $z\in \Omega$ by a uniformly convergent
power series with coefficients in $X.$

For the definition and more important properties of vector-valued distribution spaces see
\cite{a43}, \cite{fat1}, \cite{komatsu}, \cite{knjigah}, \cite{martinez}, \cite{meise},
 \cite{sch16} and references cited therein.
The Schwartz spaces of test functions are defined by $\mathcal{D}=C_0^{\infty}(\mathbb{R})$
and $\mathcal{E}=C^{\infty}(\mathbb{R})$
with usual inductive limit topologies.
The topology
of the space of rapidly decreasing functions $\mathcal{S}$ defines the following system of seminorms
$$
p_{m,n}(\psi):=\sup_{x\in\mathbb{R}}\, \bigl|x^m\psi^{(n)}(x)\bigr|,\quad
\psi\in\mathcal{S},\ m,\ n\in\mathbb{N}_0.
$$
We denote by $\mathcal{D}_{0}$ the subspace of $\mathcal{D}$ consisting of those functions $\varphi \in \mathcal{D}$ for which supp$(\varphi) \subseteq [0,\infty)$.
The spaces
$\mathcal{D}'(E):=L(\mathcal{D},E)$,
$\mathcal{E}'(E):=L(\mathcal{E},E)$ and
$\mathcal{S}'(E):=L(\mathcal{S},E)$
are topologized in the very obvious way.
$\mathcal{D}'_{0}(E)$,
$\mathcal{E}'_{0}(E)$ and $\mathcal{S}'_{0}(E)$ denote the subspaces of
$\mathcal{D}'(E)$, $\mathcal{E}'(E)$ and $\mathcal{S}'(E)$,
respectively, containing $E$-valued distributions
whose supports are contained in $[0,\infty)$.

Let $(M_p)$ be a sequence of positive numbers. We impose following conditions on $(M_p)$:\\
$(M.1)$ (Logarithmic convexity) $M^2_{p}\leq M_{p-1}M_{p+1}$ for $p\in\NN$;\\
$(M.2)$ (Stability under ultradifferential operators) For some $A,H\geq0$
$$M_p\leq AH^p\min_{0\leq q\leq p}M_{p-q}M_q,\quad p,q\in\NN;$$
$(M.3)$ (Strong non-quasi-analyticity) $$\sum\limits_{p=q+1}^{\infty}\frac{M_{p-1}}{M_p}\leq Aq\frac{M_q}{M_{q+1}},\quad q\in\NN.$$
$(M.2)'$ (Stability under differential operators) For some $A,H>1$ $$M_{p+1}\leq AH^{p+1}M_p,\quad p\in\NN;$$
$(M.3)'$ (Non-quasi-analyticity) $$\sum\limits_{p=1}^{\infty}\frac{M_{p-1}}{M_p}<\infty.$$
Let $(M_p)$ be a sequence of positive numbers satisfying $(M.1)$, $(M.2)$ and $(M.3)$. The spaces of Beurling,
respectively, Roumieu ultradifferentiable functions are
defined by $\mathcal{D}^{(M_p)}:=\mathcal{D}^{(M_p)}(\mathbb{R})
:=\text{indlim}_{K\Subset\Subset\mathbb{R}}\mathcal{D}^{(M_p)}_K$,
respectively,
$\mathcal{D}^{\{M_p\}}:=\mathcal{D}^{\{M_p\}}(\mathbb{R})
:=\text{indlim}_{K\Subset\Subset\mathbb{R}}\mathcal{D}^{\{M_p\}}_K$,
where
$\mathcal{D}^{(M_p)}_K:=\text{projlim}_{h\to\infty}\mathcal{D}^{M_p,h}_K$,
respectively, $\mathcal{D}^{\{M_p\}}_K:=\text{indlim}_{h\to 0}\mathcal{D}^{M_p,h}_K$,
\begin{align*}
\mathcal{D}^{M_p,h}_K:=\bigl\{\phi\in C^{\infty}(\mathbb{R}): \text{supp}(\phi) \subseteq K,\;\|\phi\|_{M_p,h,K}<\infty\bigr\}
\end{align*}
and
\begin{align*}
\|\phi\|_{M_p,h,K}:=\sup\Biggl\{\frac{h^p\bigl|\phi^{(p)}(t)\bigr|}{M_p} : t\in K,\;p\in\mathbb{N}_0\Biggr\}.
\end{align*}
The asterisk $*$ stands for the Beurling case $(M_p)$ or
for the Roumieu case $\{M_p\}$. Let $\emptyset \neq \Omega \subseteq {\mathbb R}.$
The space
$\mathcal{D}'^*(E):=L(\mathcal{D}^*, E)$ is the space of all continuous linear mappings from $\mathcal{D}^*$ into $E$; $\mathcal{D}^{*}_{0}$
denotes the
subspace of $\mathcal{D}^*$ containing ultradifferentiable functions of $\ast$-class whose supports are compact subsets of
$[0,\infty)$, while the symbol $\mathcal{E}'^{*}_{0}$ denotes the space consisting of all
scalar valued ultradistributions of $\ast$-class whose supports are compact subsets of
$[0,\infty)$. Similarly are defined $E$-valued analogues of the spaces $\mathcal{D}^{*}_{0}$ and $\mathcal{E}'^{*}_{0}$.
An entire function of the form
$P(\lambda)=\sum_{p=0}^{\infty}a_p\lambda^p$,
$\lambda\in\mathbb{C}$ is of class $(M_p)$, respectively, of
class $\{M_p\}$, if there exist $l>0$ and $C>0$, respectively, for every
$l>0$ there exists a constant $C>0$, such that $|a_p|\leq Cl^p/M_p$,
$p\in\mathbb{N};$ cf. \cite{k91} for further information.

\section{(Q-)Exponential C-distribution semigroups and (q-)exponential
C-ultradistribution semigroups in locally convex spaces}\label{qjua-a}
We recall the definition of the $C$-distribution semigroups and $C$-ultradistribution semigroups \cite{cdscuds}:

Let $C\in l(E)$ be and injective operator, $\mathcal{G}\in\mathcal{D}_0'(L(E))$ ($\mathcal{G}\in\mathcal{D}_0'^{\ast}(L(E))$) satisfy $C\mathcal{G}=\mathcal{G}C,$
and let $\mathcal{G}$ be boundedly equicontinuous, i.e. if for every continuous seminorm $p$ in $E$ and bounded set $B$ in $\DD^{\ast}$  there exists a continuous seminorm $q$ in $E$ such that $p(G(\varphi)x)\leq q(x)$, for any $\varphi\in B$ and $x\in E$. \\
\indent Then it is said that $\mathcal{G}$ is a pre-(C-DS) (pre-(C-UDS) of $\ast$-class) iff the following holds:
\[\tag{C.S.1}
\mathcal{G}(\varphi*_0\psi)C=\mathcal{G}(\varphi)\mathcal{G}(\psi),\quad \varphi,\;\psi\in\mathcal{D} \ \ (\varphi,\;\psi\in\mathcal{D}^{\ast}).
\]
If, additionally,
\[\tag{C.S.2}
\mathcal{N}(\mathcal{G}):=\bigcap_{\varphi\in\mathcal{D}_0}N(\mathcal{G}(\varphi))=\{0\} \ \  \Biggl(\mathcal{N}(\mathcal{G}):=\bigcap_{\varphi\in\mathcal{D}^{\ast}_0}N(\mathcal{G}(\varphi))=\{0\}\Biggr),
\]
then $\mathcal{G}$ is called a $C$-distribution semigroup ($C$-ultradistribution semigroup of $\ast$-class), (C-DS) ((C-UDS)) in short.
A pre-(C-DS) $\mathcal{G}$ is called dense if
\[\tag{C.S.3}
\mathcal{R}(\mathcal{G}):=\bigcup\limits_{\varphi\in\mathcal{D}_0}R(\mathcal{G}(\varphi))
\
\Biggl(\mathcal{R}(\mathcal{G}):=\bigcup\limits_{\varphi\in\mathcal{D}^{\ast}_0}R(\mathcal{G}(\varphi))
E \Biggr)
\
\] is dense in $E$.
The notion of a dense pre-(C-UDS) $\mathcal{G}$ of $\ast$-class (and the set $\mathcal{R}(\mathcal{G})$) is defined similarly.
The notion of a (q-)exponential
$C$-distribution semigroup ($C$-ultradistribution semigroup of $\ast$-class) is introduced as follows.
\begin{defn}(Recall \cite{cdscuds})
Let ${\mathcal G}$ be a $C$-distribution semigroup ($C$-ultradistribution semigroup of $\ast$-class).
Then $\mathcal{G}$ is said to be an
exponential $C$-distribution semigroup ($C$-ultradistribution semigroup of $\ast$-class) if there exists $\omega\in\mathbb{R}$ such that $e^{-\omega t}\mathcal{G}
\in\mathcal{S}'(L(E))$ ($e^{-\omega t}\mathcal{G}
\in\mathcal{S}'^{\ast}(L(E))$).
We use the shorthand (C-EDS) ((C-EUDS)) to denote an exponential
$C$-distribution semigroup ($C$-ultradistribution semigroup of $\ast$-class).\end{defn}
Following (\cite{segedin}) we give the next definition:
\begin{defn}\label{1145} Then $\mathcal{G}$ is said to be a quasi-equicontinuous exponential\\ (short, (q-)exponential) $C$-distribution semigroup ($C$-ultradistribution semigroup of $\ast$-class) if for every $p\in\circledast$ and bounded subset $B$ in $E$ there exist $M_p\geq1$, ${\omega}_p\geq0$ and $q_p$ seminorm on ${\mathcal S}({\mathbb R})$ $({\mathcal S}^{\ast}({\mathbb R}))$ such that $\sup\limits_{x\in B}p({\mathcal G}({\varphi})x))\leq M_p e^{{\omega}_p\cdot}q_p(\varphi)$, for all $\varphi\in{\mathcal S}_0({\mathbb R})\,\, (\varphi\in{\mathcal S}^{\ast}_0({\mathbb R}))$.
We use the shorthand q-(C-EDS) (q-(C-EUDS)) to denote a $q$-exponential
$C$-distribution semigroup ($C$-ultradistribution semigroup of $\ast$-class).\end{defn}

It should be noted that we can similarly introduce the notion of a (q-)exponential pre-(C-DS) (pre-(C-UDS) of $\ast$-class). The corresponding classes will be denoted by
pre-(C-EDS),  pre-(C-EUDS), pre-q-(C-EDS) and pre-q-(C-EUDS).\\
\indent Let $\alpha\in(0,\infty)$, $\alpha$ is not an element from ${\mathbb N}$ and $f\in{\mathcal S}$. Put $n=\lceil\alpha\rceil=\inf\{k\in{\mathbb Z}\, :\, k\geq\alpha\}$. The Weyl fractionaal derivatives (\cite{mija}), $W_+^{\alpha}$ and $W_-^{\alpha}$ of order $\alpha$ are defined by:\\
$$W_+^{\alpha}f(t):=\frac{(-1)^n}{\Gamma(n-\alpha)}\frac{d^n}{dt^n}\int\limits_t^{\infty}(s-t)^{n-\alpha-1}f(s)\, ds,\quad\quad t\in{\mathbb R},$$
$$W_-^{\alpha}f(t):=\frac{(-1)^n}{\Gamma(n-\alpha)}\frac{d^n}{dt^n}\int\limits_{-\infty}^t(s-t)^{n-\alpha-1}f(s)\, ds,\quad\quad t\in{\mathbb R}.$$
If $\alpha=n$, put $W_+^n:=(-1)^n\frac{d^n}{dt^n}$ and $W_-^n:=\frac{d^n}{dt^n}$. Then $W_{\pm}^{\alpha+\beta}=W_{\pm}^{\alpha}W_{\pm}^{\beta}$, $\alpha,\beta>0$.

\begin{thm}\label{miana-exp}
Assume that $\alpha \geq 0$ and $A$ generates a (q-)exponentially equicontinuous $\alpha$-times integrated $C$-semigroup $(S_{\alpha}(t))_{t\geq 0}.$
Then $A$ is the integral generator of a (q-)exponential (C-DS) ${\mathcal G}_{\alpha}$ defined through $\mathcal{G}_{\alpha}(\varphi)x:=\int^{\infty}_0W^{\alpha}_+\varphi(t)S_{\alpha}(t)x\,dt,\quad x\in E,\ \varphi\in\mathcal{D}$.
\end{thm}
\begin{proof}
We first show that ${\mathcal G}_{\alpha}$ is (q-)exponential semigroup. Using that $(S_{\alpha}(t))_{t\geq0}$ is (q-)exponential semigroup, we have that for all $p\in\circledast$, there exist $M_p\geq1$, $\omega_p\geq0$ and $q_p$ seminorm in ${\mathcal S}({\mathbb R})$ such that
$$p{\Big(}e^{-{\omega}_p\cdot}{\mathcal G}_{\alpha}({\varphi})x{\Big)}\leq C_1q_p({\varphi})\int_0^{\infty}\frac{dt}{{\langle x\rangle}^{\beta}}\leq Cq_p({\varphi})$$, where $\beta\geq0$ and for all $\varphi\in{\mathcal S}_0$. This one and the proof of \cite[Theorem 3.1.14]{knjigah} and \cite[Theorem 3.15]{cdscuds} give us the statement of the theorem.
\end{proof}

\begin{rem}\label{s-t}
\begin{itemize}
\item[(i)]
Suppose that $\mathcal{G}$ is a (q-)exponential $C$-distribution semigroup generated by $A,$ $\omega\in\mathbb{R}$ and $e^{-\omega t}\mathcal{G}
\in\mathcal{S}'(L(E)),$ if ${\mathcal G}$ is an (E-CDS), resp. for each $p\in \circledast$ there exists $\omega_{p}\in\mathbb{R}$ such that $e^{-\omega_{p} t}\mathcal{G}
\in\mathcal{S}'(L(E)),$ if ${\mathcal G}$ is a q-(E-CDS). Then the proof of implication (i) $\Rightarrow$ (ii) of \cite[Theorem 3.1.14]{knjigah}
shows that $e^{-\omega t}\mathcal{G}
\in\mathcal{S}'(L(E,[D(A)])),$ resp. $e^{-\omega_{p} t}\mathcal{G}
\in\mathcal{S}'(L(E,[D(A)])).$ Suppose, further, that there exists a non-negative integer $n$ and a continuous function $V : {\mathbb R} \rightarrow L(E,[D(A)])$ satisfying that
\begin{equation}\label{qju}
\bigl \langle e^{-\omega t}\mathcal{G},\varphi \bigr \rangle =(-1)^{n}\int_{-\infty}^{\infty}\varphi^{(n)}(t)V(t)\, dt,\quad \varphi \in {\mathcal D},
\end{equation}
if ${\mathcal G}$ is an (E-CDS), resp., satisfying that (\ref{qju}) holds with the number $\omega$ replaced by $\omega_{p},$
and that (in both cases) for each $p\in \circledast$ there exist $c_{p}>0,$ $r_{p}\geq 0$ and $q_{p}\in \circledast$ such that $p(V(t)x)\leq c_{p}t^{r_{p}}q_{p}(x),$ $t\geq 0,$ $x\in E.$ Then $A$ generates a (q-)exponentially equicontinuous
$n$-times integrated $C$-semigroup on $E.$
\item[(ii)] Let $a>0$, $b>0$, $\alpha>0$ and $E(a,b)\subseteq \rho_C(A).$ Suppose that the mapping $\lambda\mapsto(\lambda-A)^{-1}Cx$,
$\lambda\in E(a,b)$ is continuous for every fixed element $x\in E$, as well as that the operator family $\{(1+|\lambda|)^{-\alpha}(\lambda- A)^{-1}C : \lambda \in E(a,b)\}\subseteq L(E)$ is equicontinuous and there exists a number $a'\in (0,a)$ such that the operator family $\{(1+|\lambda|)^{-\alpha}(\lambda- A)^{-1}C : \Re \lambda >a'\}\subseteq L(E)$ is equicontinuous. Then the operator $C^{-1}AC$ generates a (C-DS) ${\mathcal G},$ (see \cite{cdscuds}), given by \begin{align}\label{franci}
\mathcal{G}(\varphi)x:=(-i)\int_{\Gamma}\hat{\varphi}(\lambda)(\lambda-A)^{-1}Cx\,d\lambda,
\;\;x\in E,\;\varphi\in\mathcal{D},
\end{align}
with $\Gamma$ being the upwards oriented boundary of region $E(a,b)$. The operator $C^{-1}AC$ generates an exponential (C-DS) ${\mathcal G},$ if $(a,\infty)\subset{\rho}_{C}(A)$ with $\Gamma$ being the line connecting the points $\bar{a}-i\infty$ and $\bar{a}+i\infty$.

\end{itemize}
\end{rem}

\begin{thm}\label{tempera-ultra}
\begin{itemize}
\item[(i)] Suppose that there exist $l>0,$ $\beta>0$ and $k>0,$ in the Beurling case,
resp., for every $l>0$ there exists $\beta_l>0$, in the Roumieu case, such that
$
\Omega^{(M_p)}_{l,\beta}:=\{\lambda\in\mathbb{C}:\Re\lambda\geq M(l|\lambda|)+\beta\}\subseteq\rho_{C}(A),
$ resp.
$
\Omega^{\{M_p\}}_{l,\beta_l}:=\{\lambda\in\mathbb{C}:\Re\lambda\geq M(l|\lambda|)+\beta_l\}\subseteq\rho_{C}(A),
$
the mapping $\lambda \rightarrow (\lambda-A)^{-1}Cx,$ $\lambda \in \Omega^{(M_p)}_{l,\beta},$ resp. $\lambda \in \Omega^{\{M_p\}}_{l,\beta_{l}},$ is continuous for every fixed element $x\in E,$
and the operator family
$
\{e^{-M(kl|\lambda|)}(\lambda-A)^{-1}C : \lambda\in\Omega^{(M_p)}_{l,\beta}\}\subseteq L(E),
$ resp. $
\{e^{-M(l|\lambda|)}(\lambda-A)^{-1}C : \lambda\in\Omega^{\{M_p\}}_{l,\beta_{l}}\}\subseteq L(E),
$
is equicontinuous. Denote by $\Gamma ,$ resp. $\Gamma_{l},$ the upwards oriented boundary of $\Omega^{(M_p)}_{l,\beta},$ resp. $\Omega^{\{M_p\}}_{l,\beta_{l}}.$ Define, for every $x\in E$ and $\varphi\in\mathcal{D}^{\ast},$ the element
$
{\mathcal G}(\varphi)x$ with
\begin{align} \mathcal{G}(\varphi)x:=(-i)\int_{\Gamma}\hat{\varphi}(\lambda)(\lambda-A)^{-1}Cx\,d\lambda,
\;\;x\in E,\;\varphi\in\mathcal{D},
\end{align} in the Beurling case; in the Roumieu case, for every number $k>0$ and for every function $\varphi \in {\mathcal D}^{\{M_{p}\}}_{[-k,k]},$
we define the element ${\mathcal G}(\varphi)x$ in the same way as above, with the contour $\Gamma$ replaced by $\Gamma_{l(k)}$.
Then ${\mathcal G}\in {\mathcal D}^{\prime \ast}_{0}(L(E))$ is boundedly equicontinuous,
${\mathcal G}(\varphi)C=C{\mathcal G}(\varphi),$ $\varphi \in {\mathcal D}^{\ast},$
$\mathcal{G}(\varphi)A\subseteq A{\mathcal G}(\varphi),$ $\varphi \in {\mathcal D}^{\ast}$ and $A\mathcal{G}(\varphi)x=\mathcal{G}\bigl(-\varphi'\bigr)x-\varphi(0)Cx,\quad x\in E,\ \varphi \in {\mathcal D} \ \ (\varphi \in {\mathcal D}^{\ast})$. In particular, ${\mathcal G}$ is a pre-(C-UDS) of $\ast$-class.
If ${\mathcal G}$ additionally satisfies \emph{(C.S.2)}, then ${\mathcal G}$ is a (C-UDS) of $\ast$-class generated by $C^{-1}AC.$
\item[(ii)]
Suppose that $A$ is a closed linear operator on $E$ satisfying that there exist $a\geq 0$ such that $
\{\lambda\in\mathbb{C}:\Re\lambda>a\}\subseteq\rho_{C}(A)$ and the mapping $\lambda \mapsto (\lambda-A)^{-1}Cx,$ $\Re \lambda>a$ is continuous for every fixed element $x\in E.$
Suppose that there exists a number $k>0,$ in the Beurling case, resp., for every number $k>0,$
in the Roumieu case,
such that the operator family
$\{e^{-M(k|\lambda|)}(\lambda-A)^{-1}C : \Re \lambda>a\}\subseteq L(E)$ is equicontinuous. Set
$$
{\mathcal G}(\varphi)x=(-i)\int\limits_{\bar{a}-i\infty}^{\bar{a}
+i\infty}\hat{\varphi}(\lambda)\bigl(\lambda-A\bigr)^{-1}Cx\,d\lambda,\quad x\in E,\ \varphi\in\mathcal{D}^{\ast}.
$$
Then ${\mathcal G}\in {\mathcal D}^{\prime \ast}_{0}(L(E))$ is boundedly equicontinuous, $e^{-\omega t}G\in\mathcal{S}^{\prime \ast}(L(E))$ for all $\omega >a,$
${\mathcal G}(\varphi)C=C{\mathcal G}(\varphi),$ $\varphi \in {\mathcal D}^{\ast},$
$\mathcal{G}(\varphi)A\subseteq A{\mathcal G}(\varphi),$ $\varphi \in {\mathcal D}^{\ast}$ and $A\mathcal{G}(\varphi)x=\mathcal{G}\bigl(-\varphi'\bigr)x-\varphi(0)Cx,\quad x\in E,\ \varphi \in {\mathcal D} \ \ (\varphi \in {\mathcal D}^{\ast})$. In particular, ${\mathcal G}$ is a pre-(C-EUDS) of $\ast$-class.
If ${\mathcal G}$ additionally satisfies \emph{(C.S.2)}, then ${\mathcal G}$ is a (C-EDS) of $\ast$-class generated by $C^{-1}AC.$
\end{itemize}
\end{thm}
\begin{proof}
Using the Paley-Wiener theorem for ultradistributions \cite[Theorem 9.1]{k91} we have that there exist $k$, $h$ and $m$ (each number $k>0$, we can find constants $m(k)>0$ and $h(k)>0$) such that $|\hat{\varphi}|\leq m(k) e^{-M(\frac{|\lambda|}{h(k)})+k|\Re\lambda|}$, $\lambda\in{\mathbb C}$, $\varphi\in{\mathcal D}_K^{\ast}$. By the Cauchy theorem we can deform the path of integration to the straight line connecting $\bar{a}-i\infty$ and $\bar{a}+i{\infty}$, where $\bar{a}\in(0,\beta)$. Now we will show that ${\mathcal G}$, defined with ${\mathcal G}({\varphi})x=(-i)\int\limits_{\Gamma}\hat{\varphi}({\lambda})(\lambda-A)^{-1}Cx\, d\lambda$, $x\in E$, $\varphi\in{\mathcal D}^{\ast}$, where $\hat{\varphi}(\lambda)=\int_{-\infty}^{+\infty}e^{\lambda t}\varphi(t)\, dt$, $\lambda\in{\mathbb C}$, $\varphi\in{\mathcal D}^{\ast}$, is an $C$-ultradistribution fundamental solution for $A$. Then the existence of such ${\mathcal G}$ and further results on ${\mathcal G}$ are consequences of Theorem 2.5, Theorem 2.8 and Theorem 3.6 from \cite{cdscuds}. We will give the proof for Beurling case. The consideration in Roumieu case is quite similar. By assumptions we have that for every $p\in\circledast$, there exist $C_p\geq1$ and $q_p\in\circledast$ such that $p((\lambda-A)^{-1}Cx)\leq C_pe^{\Re\lambda-\beta}q_p(x)$, for $\lambda\in\Gamma$. Let $\mbox{supp}\,\varphi(t)\subset K=[a,b]$. Then,
$$|\hat{\varphi}(\lambda)|={\Big|}\frac{(-1)^n}{{\lambda}^n}\int\limits_Ke^{\lambda t}{\varphi}^{(n)}(t)\, dt{\Big|}\leq\sup\limits_n\frac{\|{\varphi}^{(n)}\|}{M_nh^n}\cdot\frac{M_nh^n}{|\lambda|^n}\int\limits_Ke^{\Re(\lambda t)}\, dt=$$
$$=\|\varphi\|_{M_n, h, K}\cdot\frac{M_nh^n}{|\lambda|^n}\cdot\frac{e^{b\Re\lambda}-e^{a\Re\lambda}}{\Re\lambda},$$ for all $h>0$. Since $\inf\limits_n\frac{M_nh^n}{|\lambda|^n}=e^{-M(\frac{\lambda}{h})}$, we have $$|\hat{\varphi}(\lambda)|\leq C_1\|\varphi\|_{M_n,h,K}e^{b\Re\lambda-M(\frac{\lambda}{h})}.$$ Moreover,  $M(\frac{\lambda}{h})=\Re\lambda{\Big(}\frac{1}{\alpha h}-\beta{\Big)}$ on $\Gamma$, implies that for every $p\in\circledast$, $$p({\mathcal G}({\varphi})x)\leq\int\limits_{\Gamma}p((\lambda-A)^{-1}C)|\hat{\varphi}(\lambda)|\,|d\lambda|\leq$$
$$\leq\int\limits_{\Gamma}C_pC_1e^{(\Re\lambda-\beta)+b\Re\lambda-M(\frac{\lambda}{h})}\|\varphi\|_{M_n,h,K}\, |d\lambda|=\int\limits_{\Gamma}C_2e^{(1+b-\frac{1}{\alpha h})\Re\lambda}\|\varphi\|_{M_n,h,K}\, |d\lambda|.$$
Now, for any $\alpha$ and $b$ we can find $h$ such that $1+b-\frac{1}{\alpha h}<0$, hence ${\mathcal G}({\varphi})$ is well defined. Now we will show that ${\mathcal G}$ is a $C$-ultradistribution fundamental solution for $A$. We put $P={\delta}'\otimes I_E-\delta\otimes A$. Then,
$$(P\ast {\mathcal G})({\varphi})={\Big(}({\delta}'\otimes I_E-\delta\otimes A)\ast{\mathcal G}{\Big)}({\varphi})=({\mathcal G}'-A{\mathcal G})({\varphi})={\mathcal G}(-{\varphi}')-A{\mathcal G}({\varphi}).$$ This gives
$$(P\ast{\mathcal G})({\varphi})=(-i)\int\limits_{\Gamma}(\lambda-A)^{-1}C\int\limits_{-\infty}^{+\infty}e^{\lambda t}{\varphi}(t)\, dt\, d\lambda-A{\mathcal G}({\varphi})=i\int\limits_{\Gamma}\int\limits_Ke^{\lambda t}\varphi(t)C\, dt\, d\lambda.$$
We can deform $\Gamma$ into a straight line connecting $\bar{a}-i\infty$ and $\bar{a}+i\infty$, so $$(P\ast{\mathcal G})({\varphi})=i\int\limits_{\bar{a}-i\infty}^{\bar{a}+i\infty}\int\limits_Ke^{\lambda t}\varphi(t)C\, dt\, d\lambda=i\int\limits_{\bar{a}-i\infty}^{\bar{a}+\infty}\int\limits_K\frac{e^{\lambda t}}{{\lambda}^2}{\varphi}''(t)C\, dt\, d\lambda.$$
The last integral is an inverse Laplace transform of the identity,  $$\langle{\mathcal L}^{-1}(1),\varphi\rangle=\langle i\int\limits_{\bar{a}-i\infty}^{\bar{a}+i\infty}\frac{e^{\lambda t}}{{\lambda}^2}\, d\lambda, {\varphi}''\rangle=\langle\delta,\varphi\rangle.$$ Thus, we obtain that $(P\ast{\mathcal G})({\varphi})=\langle\delta,\varphi\rangle\, C$ and ${\mathcal G}$ is a $C$-ultradistribution fundamental solution for $A$.\\
The proof of ii) is similar, so we omit it.

\end{proof}

\begin{rem}\label{ultra-logarithmic}
\begin{itemize}
\item[(i)]
Suppose that there exists an integer $n\in {\mathbb N}$ such that the family $\{(1+|\lambda|)^{-n}(\lambda-A)^{-1}C : \lambda \in \Omega^{(M_p)}_{l,\beta}\}$ ($\{(1+|\lambda|)^{-n}(\lambda-A)^{-1}C : \lambda \in \Omega^{(M_p)}_{l,\beta_{l}}\}$) in the case of (i),
and the family
$\{(1+|\lambda|)^{-n}(\lambda-A)^{-1}C : \Re\lambda >a\},$ in the case of (ii), is equicontinuous. Then ${\mathcal G}$ satisfies (C.S.2), and consequently,  ${\mathcal G}$ is a (C-UDS) of $\ast$-class ((C-EUDS) of $\ast$-class, in the case of (ii)) generated by $C^{-1}AC.$ The proof of this fact can be given as in the Banach space case (\cite{ci1}, \cite{knjigah}).
\item[(ii)]  Following J. Chazarain, we define the $(M_p)$-ultralogarithmic region $\Lambda_{\alpha,\beta,l}$ of type $l$ as follows
$$
\Lambda_{\alpha,\beta,l}:=\{\lambda\in\mathbb{C}:\Re\lambda\geq\alpha M(l|\Im\lambda|)+\beta\} \ \ (\alpha>0,\ \beta>0,\ l\in {\mathbb R}).
$$
Then the first part of Theorem \ref{tempera-ultra} can be reformulated with the region $\Omega^{(M_p)}_{l,\beta}$ replaced by $
\Lambda_{\alpha,\beta,l}.$
\end{itemize}
\end{rem}


In the remaining part of this section, we are interested in the following problem:

\begin{itemize}
\item[(Q)]
Suppose that a closed linear operator $A$ satisfies that there exist constants $l\geq 1$, $\alpha>0,$ $\beta>0$ and $k>0$
such that $
\Lambda_{\alpha,\beta, l}\subseteq\rho (A)$ ($RHP_{\alpha}\equiv \{\lambda \in {\mathbb C} : \Re \lambda>\alpha\}\subseteq \rho (A)$)  and that for each seminorm $q\in \circledast$ there exists a number $c_{q}>0$ and a seminorm $r\in \circledast$ such that
\begin{equation}\label{but-surf}
q\Bigl( \bigl(\lambda-A\bigr)^{-1}x \Bigr)\leq c_{q}e^{M(kl|\lambda|)}r(x),\quad x\in E,\ \lambda \in \Lambda_{\alpha,\beta, l} \ \ \bigl(RHP_{\alpha} \bigr).
\end{equation}
Does
there exist an injective operator $C\in L(E)$
such that $A$ generates a global locally equicontinuous
(exponentially equicontinuous) $C$-regularized semigroup $(S(t))_{t\geq 0}$ on $E?$
\end{itemize}

Let $\bar{\alpha}>\alpha.$ Designate by $\Gamma_{l}$ ($\Gamma_{\bar{\alpha}}$) the upwards oriented boundary of the ultra-logarithmic region $\Lambda_{\alpha,\beta, l}$ (the right line connecting the points $\bar{\alpha}-i\infty$ and $\bar{\alpha}+i\infty$).
Put
\begin{equation}\label{polugrupa}
{\mathcal G}(\varphi)x:=(-i)\int_{\Gamma_{l}\ (\Gamma_{\bar{\alpha}})}\hat{\varphi}(\lambda)(\lambda-A)^{-1}x\, d\lambda,\quad x\in E,\quad \varphi \in {\mathcal D}^{(M_{p})}.
\end{equation}
Then, Theorem \ref{tempera-ultra} implies that ${\mathcal G}$ is a pre-(UDS) (pre-(EUDS)) of Beurling class.
We will answer the question (Q) in the affirmative provided that ${\mathcal G}$ is a (UDS) ((EUDS)) of Beurling class generated by $A$ (which is equivalent to say that ${\mathcal G}$ satisfies (C.S.2) with $C=I$), as well as that  $(M_{p})$ satisfies (M.1), (M.2) and (M.3). In the remaining part of this section, these conditions will be assumed to hold true.

The following entire function of exponential type zero (\cite{k91}) plays a crucial role in the proof of Theorem \ref{grace} below:
$\omega(z)=:\prod_{i=1}^{\infty}\bigl(1+\frac{iz}{m_p}\bigr)$, $z\in\mathbb{C}$. Then
the properties (P.1)-(P.5) stated in \cite[Subsection 3.6.2]{knjigah} hold.
We define the abstract Beurling space of $(M_p)$ class
associated to a closed linear operator $A$ as in \cite{ci1}.
Put $E^{(M_p)}(A):=$projlim$_{h\to+\infty}E^{(M_p)}_h(A)$, where
$$
E^{(M_p)}_h(A):=\Biggl\{x\in D_{\infty}(A):\|x\|^{(M_p)}_{h,q}=\sup_{p\in\mathbb{N}_0}\frac{h^{p}q\bigl(A^px\bigr)}{M_p}<\infty\mbox{ for all }h>0\mbox{ and } q\in \circledast\Biggr\}.
$$
Then, for each number $h>0$ the calibration $(\|\cdot\|^{(M_p)}_{h,q})_{q\in \circledast}$ induces a Hausdorff sequentially complete locally convex space on $E^{(M_p)}_h(A),$
$E^{(M_p)}_{h'}(A)\subseteq E^{(M_p)}_h(A)$
provided $0<h<h'<\infty ,$  and the spaces  $E^{(M_p)}_h(A)$  and  $E^{(M_p)}(A)$ are continuously  embedded in $E$.
As in the Banach space case (\cite{knjigah}), we have that
$E^{(M_p)}(A)=E^{(M_p)}(A-z)$, $z\in\mathbb{C}$ and that the part of $A$ in $E^{(M_p)}(A)$ is a linear continuous mapping from $E^{(M_p)}(A)$ into $E^{(M_p)}(A)$.
It can be simply verified that ${\mathcal R}({\mathcal G})\subseteq E^{(M_{p})}(A)$ (cf. (\ref{polugrupa})).

\begin{thm}\label{grace}
Suppose that the requirements of (Q) hold, as well as that ${\mathcal G},$ defined through \emph{(\ref{polugrupa})}, is a (UDS) ((EUDS)) of Beurling class generated by $A$ (i.e., that ${\mathcal G}$ satisfies \emph{(C.S.2)}), and that  $(M_{p})$ satisfies
\emph{(M.1)}, \emph{(M.2)} and \emph{(M.3)}.
\begin{itemize}
\item[(i)]
Then the abstract Cauchy problem $(ACP)$
has a unique solution $u(t)$ for all $x\in E^{(M_p)}(A)$;
furthermore, the mapping $t\mapsto u(t)\in E,$ $t\geq 0$ is infinitely differentiable. If $K$ is a compact subset of $[0,\infty)$ and $h>0$, then the solution $u(t)$ of $(ACP)$ satisfies
$$
\sup_{t\in K,p\in\mathbb{N}_0}\frac{h^p}{M_p}q\Biggl(\frac{d^p}{dt^p}u(t)\Biggr)<\infty,\quad q\in \circledast.
$$
\item[(ii)]
There exists an injective operator $C\in L(E)$
such that $A$ generates a global locally equicontinuous (exponentially equicontinuous) $C$-regularized semigroup $(S(t))_{t\geq 0}\subseteq L(E)$ such that the mapping $t\mapsto S(t)\in L(E),$ $t\geq 0$ is
infinitely differentiable and $E^{(M_{p})}(A)\subseteq C(D_{\infty}(A)).$
Furthermore, for every compact set $K\subseteq[0,\infty),$ for every $h>0$, and for every bounded subset $B$ of $E,$ one has
$$
\sup_{t\in K,x\in B,p\in\mathbb{N}_0}\frac{h^p}{M_p}q\Biggl(\frac{d^p}{dt^p}S(t)x\Biggr)<\infty ,\quad q\in \circledast.
$$
\end{itemize}
\end{thm}
\begin{proof}
\begin{itemize}
\item[(i)]
The uniqueness of the solution of $(ACP)$ follows by the use of the Ljubich uniqueness type theorem \cite[Lemma 2.1.33(ii)]{knjigah}. Let $$S(t)=\frac{1}{2\pi i}\int\limits_{{\Gamma}_l} e^{{\lambda}t}\frac{R({\lambda}:A)}{{\omega}^{nn_0}(i{\lambda})}\, d{\lambda}.$$ Moreover, we have $$|{\omega}(s)|=\prod\limits_{k=1}^{\infty}|1+\frac{is}{m_k}|\geq\sup\limits_{p\in{\mathbb N}}\prod\limits_{k=1}^p\frac{s}{m_k}=\sup\limits_{p\in{\mathbb N}}\frac{s^p}{M_p}\geq e^{M(s)}, s>0.$$
We put $S(0)=C_{\tau}\in L(E)$. Since ${\mathcal G}$ satisfies \emph{(C.S.2)}), $C_{\tau}$ is injective. We have $(n-1)n_0{\sigma}\geq(n-1)Hll_0>n-1>1$ and $(n-1)n_0{\sigma}l_0^{-1}\geq(H^k+1)Hl_0l{\sigma}^{-1}{\sigma}l_0^{-1}>H^kl,$ where $n_0=\lfloor Hl_0l{\sigma}^{-1}\rfloor+1$, $k=\max(\lceil{\tau}{\alpha}\rceil,\alpha)$, $n_0\sigma>1$. Thus, for every $p\in{\mathbb N}_0$ and $q\in\circledast$, $$q{\Big(}{\lambda}^p\frac{e^{\lambda t}R({\lambda}:A)}{{\omega}^{nn_0}(i{\lambda})}x{\Big)}\leq C|{\lambda}|^p\frac{1}{e^{M(|{\lambda}|(n_0{\sigma}-1))}|{\lambda}|^2}q(x)$$ where the constant $C$ is independent of $p$. By the Fubini theorem we have $S(s)C_{\tau}=C_{\tau}S(s)$. Furthermore, $S(s)A\subseteq AS(s)$. Since ${\rho}(A)\neq0$, we have $C_{\tau}^{-1}AC_{\tau}=A$. Now, we will prove that $(S(t))_{t\geq0}$ is a global locally equicontinuous $C_{\tau}$-regularized semigroup generated by $A$, i.e. we will prove that $A\int\limits_0^tS(s)x\, ds=S(t)x-C_{\tau}x$, $t\in[0,\infty)$. For this, we need to prove \begin{equation}\label{667} \int\limits_{{\Gamma}_l}\frac{e^{\lambda t}}{{\omega}^{nn_0}(i{\lambda})}\, d{\lambda}=0.\end{equation} Since $$|\frac{e^{{\lambda}t}}{{\omega}^{nn_0}(i{\lambda})}|\leq\frac{C}{R^2}c_0e^{t\beta}A^{\lceil t\alpha\rceil-1}\frac{{\omega}(H^{\lceil t\alpha\rceil-1}lR)|}{|{\omega}((n-1)n_0\sigma l_0^{-1}R)|},$$
we have $\int_{{\Gamma}_R}\frac{e^{\lambda}t}{{\omega}^{nn_0}(i{\lambda})}\, d{\lambda}\longrightarrow0$, when $R\longrightarrow\infty$. By the Cauchy theorem, we obtain (\ref{667}). Now,
$$A\int\limits_0^tS(s)x\, ds=\frac{1}{2\pi i}\int\limits_0^t\int\limits_{{\Gamma}_l}e^{{\lambda} s}\frac{AR({\lambda}:A)x}{{\omega}^{nn_0}(i{\lambda})}\, d{\lambda}\, ds=$$
$$=\frac{1}{2\pi i}\int\limits_0^t\int\limits_{{\Gamma}_l}e^{{\lambda}s}\frac{{\lambda}R({\lambda}:A)x}{{\omega}^{nn_0}(i{\lambda})}\, d{\lambda}\, ds-\frac{1}{2\pi i}\int\limits_0^t\int\limits_{{\Gamma}_l}e^{{\lambda}s}\frac{x}{{\omega}^{nn_0}(i{\lambda})}\, d{\lambda}\, ds=$$
$$=\frac{1}{2\pi i}\int\limits_{{\Gamma}_l}{\Big(}\int\limits_0^te^{{\lambda}s}\frac{\lambda R({\lambda}:A)x}{{\omega}^{nn_0}(i{\lambda})}\, ds{\Big)}\, d{\lambda}=\frac{1}{2\pi i}\int\limits_{{\Gamma}_l}(e^{\lambda t}-1)\frac{R({\lambda}:A)x}{{\omega}^{nn_0}(i{\lambda})}\, d{\lambda}=$$
$$=\frac{1}{2\pi i}\int\limits_{{\Gamma}_l}e^{\lambda t}\frac{R({\lambda}:A)x}{{\omega}^{nn_0}(i{\lambda})}\, d{\lambda}-\frac{1}{2\pi i}\int\limits_{{\Gamma}_l}\frac{R({\lambda}:A)}{{\omega}^{nn_0}(i{\lambda})}\, d{\lambda}=$$ $$=S(t)x-C_{\tau}x.$$ Thus, for every $p\in {\mathbb N}$, the integral $\frac{1}{2\pi i}\int\limits_{{\Gamma}_l}{\lambda}^pe^{{\lambda}t}\frac{R({\lambda}:A)}{{\omega}^{nn_0}(i{\lambda})}\, d{\lambda}$ is convergent and that $$\frac{d}{dt}S(t)=\frac{1}{2\pi i}\int\limits_{{\Gamma}_l}{\lambda}e^{{\lambda}t}\frac{R({\lambda}:A)}{{\omega}^{nn_0}(i{\lambda})}\, d{\lambda}.$$ By induction $$\frac{d^p}{dt^p}S(t)=\frac{1}{2\pi i}\int\limits_{{\Gamma}_l}{\lambda}^pe^{{\lambda}t}\frac{R({\lambda}:A)}{{\omega}^{nn_0}(i{\lambda})}\, d{\lambda},\,\,\,\, p\in {\mathbb N}.$$ Now, for every $q\in\circledast$

$$\sup\limits_{t\in K, p\in{\mathbb N}_0}\frac{h^p}{M_p}q{\Big(}\frac{d^p}{dt^p}S(t){\Big)}\leq\frac{1}{2\pi}\sup\limits_{t\in K, p\in{\mathbb N}_0}\frac{h^p}{M_p}\int\limits_{{\Gamma}_l}\frac{|{\lambda}|^p|e^{{\lambda}t}|q(R({\lambda}:A))}{|{\omega}^{nn_0}(i{\lambda})|}
\, |d{\lambda}|\leq$$ \begin{equation}\label{223}\leq C\int\limits_{{\Gamma}_l}\frac{e^{M(h|\lambda|)}}{e^{M(|\lambda|(n_0\sigma-1))}|\lambda|^2}\, |d{\lambda}|\leq C\int\limits_{{\Gamma}_l}\frac{|d\lambda|}{|\lambda|^2}<\infty.\end{equation}
Thus there exist a number $n_0\in{\mathbb N}_0$ and a strictly increasing sequence $(k_l)_l$ in ${\mathbb N}$ such that $n_0> Hl_0l{\sigma}^{-1}$ and for every $l\in{\mathbb N}$ the operator $A$ is the generator of  a differentiable $C_{\tau}$-regularized semigroup $(S(t))$ . This implies that $(ACP)$ has a unique solution for every $x\in C_{\tau}(D(A))$, given by $u(t)=C_{\tau}^{-1}S(t)x$. If $x\in E^{(M_p)}(A)$, then $u(t)=S(t)C_{\tau}^{-1}x$ and $u(t)\in{\mathcal C}([0,\infty),E)$. Therefore we obtain a solution of the $(ACP)$ for $x\in E^{(M_p)}(A)$. By (\ref{223}) and the form of the solutions $u(t)$ we can obtain that for all $q\in\circledast$, $$\sup_{t\in K,p\in\mathbb{N}_0}\frac{h^p}{M_p}q\Biggl(\frac{d^p}{dt^p}u(t)\Biggr)<\infty.$$

\item[(ii)] By the use of the first part of the proof of (i), with slight modification and the fact that for any $(M_p)$ there exists $(N_p)$ satisfying $N_0=1$, $(M.1)$, $(M.2)$, $(M.3)$ and $N_p\prec M_p$ we prove the assertion.

\end{itemize}
\end{proof}
\begin{rem}\label{invarijantnost}
\begin{itemize}
\item[(i)]
Suppose that (\ref{but-surf}) holds with $q=r.$ Then there exists an injective operator $C\in L_{\circledast}(E)$
such that $A$ generates a global locally equicontinuous (exponentially equicontinuous) $C$-regularized semigroup $(S(t))_{t\geq 0}\subseteq L_{\circledast}(E)$ satisfying the properties stated in Theorem \ref{grace}(ii).
\item[ii)]
Observe that the previous analysis, with insignificant modifications, can be made for the corresponding ultra-logarithmic region $
\Lambda_{\alpha,\beta, l}$ belongs to the $C'$-resolvent set of $A,$ for some injective operator $C'\in L(E)$ commuting with $A$.
\end{itemize}
\end{rem}

\section{Examples and applications}

In this section, we will provide several illustrative examples of $C$-distribution semigroups and $C$-ultradistribution semigroups in Fr\' echet spaces. We pay special attention to the generation of ultradistibution semigroups of Beurling class.

\begin{example}\label{distr-jedan}
The concrete examples of differential operators generating\\ (q-)exponentially equicontinuous (locally equicontinuous) fractionally integrated $C$-semigroups in Fr\' echet spaces can be found in
\cite{moore}, \cite{x263} and \cite{segedin}. In view of Theorem \ref{miana-exp}, these operators can be used for construction of (q-)exponential $C$-distribution semigroups.
\end{example}

\begin{example}\label{multiplication}
(\cite[Example 4.4(c)]{a22}; cf. also \cite[Example 3.2(ii)]{segedin} and Example \ref{prusinjo} below)
Put $E:=\{f\in C^{\infty}([0,\infty)) : \lim
_{x\rightarrow +\infty}f^{(k)}(x)=0\mbox{ for all }k\in {{\mathbb
N}_{0}}\}$ and $||f||_{k}:=\sum_{j=0}^{k}\sup_{x\geq
0}|f^{(j)}(x)|,$ $f\in E,$ $k\in {{\mathbb N}_{0}}.$ Then the
topology induced by these norms turns $E$ into a Fr\' echet space. Consider the densely defined operator
$A : D(A) \rightarrow E,$ where $D(A):=\{f(x)\in E : (x+ie^{x})f(x)\in E\}$ and $Af(x):=(x+ie^{x})f(x),$ $x\geq 0,$ $f\in E.$
In the case that $X=L^{p}(1,\infty)$ for some $p\in [1,\infty],$ W. Arendt, O. El-Mennaoui and V. Keyantuo have proved in \cite{a22} that the operator $A$ generates a once integrated semigroup
$(S_{1}(t))_{t\in [0,1]}$ on $X,$ given by $(S_{1}(t)f)(x):=(x+ie^{x})^{-1}(e^{t(x+ie^{x})}-1),$ $x>1,$ $t\in [0,1].$ The main purpose of this example is to show that a completely different result holds if we choose $E$ as the state space. Speaking-matter-of-factly,
we cannot find numbers $n\in {\mathbb N}$ and $\tau>0$ such that $A$ generates a local $n$-times integrated semigroup $(S_{n}(t))_{t\in [0,\tau)}$ on $E.$ Suppose to the contrary that $(S_{n}(t))_{t\in [0,\tau)}$ is
a local $n$-times integrated semigroup generated by $A.$ Then it can be simply proved that $(S_{n}(t))_{t\in [0,\tau)}$ must be given by the following formula
$$
(S_{n}(t)f)(x):=\Biggl[ \frac{e^{t(x+ie^{x})}}{(x+ie^{x})^{n}}-\frac{t^{n-1}}{(n-1)!}\frac{1}{x+ie^{x}}-\cdot \cdot \cdot -\frac{t}{(x+ie^{x})^{n-1}}-\frac{1}{(x+ie^{x})^{n}}
\Biggr]f(x),
$$
for any $f\in E,$ $x\geq 0$ and $t\in [0,\tau).$
This immediately implies that for each $t>0$ there exists $f_{t}\in E$ such that $\| S_{n}(t)f_{t} \|_{n}=+\infty ;$ hence, there is no $t>0$ such that $S_{n}(t) \in L(E).$ A contradiction.  On the other hand,
we have that any complex number $\lambda \in {\mathbb C} \setminus S,$ where $S:=\{x+ie^{x} : x\geq 0\},$ belongs to the resolvent set of $A$ and
$$
\bigl(\lambda-A\bigr)^{-1}f(x)=\frac{f(x)}{\lambda- \bigl(x+ie^{x}\bigr)},\quad \lambda \in {\mathbb C} \setminus S,\ x\geq 0.
$$
Let $s>1,$ and let $a>0,$ $b>1$ satisfy that $x -\ln (((x-b)/a)^{s}+1)\geq 1,$ $x\geq b.$ Set $\Omega :=\{\lambda \in {\mathbb C} : \Re \lambda \geq a|\Im \lambda|^{1/s}+b\}$ and denote by $\Gamma$ the upwards oriented boundary of the region $\Omega .$
Inductively, we can prove that for each number $n\in {\mathbb N}$ there exist complex polynomials $P_{j}(z)=\sum^{j}_{l=0}a_{}z^{l}$ ($1\leq j \leq n$) such that dg$(P_{j})=j,$
$|a_{j,l}|\leq (n+1)!$ ($1\leq j \leq n,$ $0\leq l\leq j$) and
\begin{equation}\label{armatura}
\frac{d^{n}}{dx^{n}}\Bigl(\lambda-\bigl(x+ie^{x}\bigr) \Bigr)^{-1}=\sum \limits^{n+1}_{j=1}\Bigl(\lambda-\bigl(x+ie^{x}\bigr) \Bigr)^{-j-1}P_{j}\bigl(e^{x}\bigr),\quad x\geq 0,\ \lambda \in {\mathbb C}\setminus S.
\end{equation}
Suppose $\lambda \in \Omega$ and $x\geq 0.$ If $|\Im \lambda -e^{x}|\geq 1,$ then we have the following estimate
\begin{align}
\notag &\frac{e^{2jx}}{\bigl(\Re \lambda -x\bigr)^{2k}+\bigl( \Im \lambda-e^{x}\bigr)^{2k}}\leq \frac{e^{2jx}}{\bigl( \Im \lambda-e^{x}\bigr)^{2k}}
\\\label{armatura1} & \leq 2^{2j}\bigl( 1+|\Im \lambda|\bigr)^{2j},\quad k\in {\mathbb N}_{0},\ 0\leq j< k.
\end{align}
If $|\Im \lambda -e^{x}|<1,$ then $\Im \lambda >0,$ $0\leq x<\ln (\Im \lambda +1),$ and
\begin{align}
\notag & \frac{e^{2jx}}{\bigl(\Re \lambda -x\bigr)^{2k}+\bigl( \Im \lambda-e^{x}\bigr)^{2k}}\leq \frac{e^{2jx}}{\bigl(\Re \lambda -x\bigr)^{2k}}
\\\label{armatura2} & \leq  \frac{\bigl(\Im \lambda +1\bigr)^{j}}{\Re \lambda -\ln \bigl(((\Re \lambda-b)/a)^{s}+1\bigr)}\leq \bigl(\Im \lambda +1\bigr)^{j},\quad k\in {\mathbb N}_{0},\ 0\leq j< k.
\end{align}
Combining (\ref{armatura})-(\ref{armatura2}), it can be simply proved that the operator family\\ $\{ e^{-d|\lambda|^{1/s}}(\lambda - A)^{-1} : \lambda \in \Omega\}\subseteq L(E)$ is equicontinuous for each number $d>0.$ Now it is not difficult to prove with the help of
Theorem \ref{tempera-ultra}(i) that $A$ generates an ultradistribution semigroup ${\mathcal G}$ of $(p!^{s})$-class, given by $({\mathcal G}(\varphi)f)(x)=(2\pi)^{-1}\hat{\varphi}(x+ie^{x})f(x),$ $f\in E,$ $x\geq 0,$ $\varphi \in {\mathcal D}^{(p!^{s})}$ (on the basis of above results, we can construct corresponding examples of (local) $K$-convoluted semigroups generated by $A$). Observe, finally, that this semigroup is not differentiable and that the same result holds if we choose ${\mathcal D}_{L^{p}(0,\infty)}$ as the state space.
\end{example}

\begin{example}\label{kunst-irish}
Let $h>0.$ Consider the  Fr\'echet space
\begin{align*}
E_{(h)}:=\Biggl\{f\in C^{\infty}& ([0,\infty))\, :\,  f(0)=0,
\\ & p_n(f):=\sup\limits_{k\in{\mathbb N}_0}\sup\limits_{t\geq k}\frac{h^n|{f}^{(n)}(t)|}{M_n}<+\infty \mbox{ for all }n\in {{\mathbb N}_{0}} \Biggr\}
\end{align*}
and a closed linear operator $A:=-d/ds$ with maximal domain in $E_{(h)}.$ Then it can be simply shown that any $\lambda \in {\mathbb C}$ with $\Re \lambda >0$ belongs to $\rho(A),$ with $(\lambda-A)^{-1}f(x)=\int^{x}_{0}e^{\lambda (s-x)}f(s)\, ds,$
$x\geq 0,$ $f\in E_{(h)}.$ Carrying out a very simple computation, we get that for each $n\in {\mathbb N}_{0}$ there exists a constant $c_{n}>0$ such that
$p_{n}((\lambda-A)^{-1}f)\leq c_{n}e^{M(h(1+|\lambda|))}(p_{0}(f)+\cdot \cdot \cdot +p_{n}(f)),$ $f\in E_{(h)},$ $\Re \lambda >0.$
By Theorem \ref{tempera-ultra}(ii), it readily follows that the operator $A$ generates a pre-(EUDS) ${\mathcal G}$ of Beurling class.
We can directly verify that ${\mathcal G}$ satisfies the condition (C.S.2) so that  ${\mathcal G}$ is, actually, an (EUDS) of Beurling class.
On the other hand, E. Borel's theorem for ultradifferentiable functions of Beurling class (see e.g. \cite{ptica}) shows that the operator $A$ is not stationary dense. By a statement in \cite{statio}, there is no $n\in {\mathbb N}$ such that the operator $A$ generates a local $n$-times integrated semigroup on $E_{(h)}.$
\end{example}

In the next example, we use the notion and terminology from \cite[Chapter 8]{a43}.

\begin{example}\label{skulls}
Suppose that $p\in[1,\infty)$, $m>0$, $\rho\in[0,1]$, $r>0$, $a\in S_{\rho,0}^m$ satisfies $(H_r),$
$E=L^p(\mathbb{R}^n)$ or $E=C_0(\mathbb{R}^n)$
(in the last case, we assume $p=\infty$),  $0\leq l\leq n,$ $A:=\mathrm{Op}_E(a)$
and
\begin{equation}\label{mh}
n\Bigl|\frac{1}{2}-\frac{1}{p}\Bigr|\Bigl(\frac{m-r-\rho+1}{r}\Bigr)<1.
\end{equation}
Let us recall that if $a(\cdot)$ is an elliptic polynomial of order $m$, then \eqref{mh} holds with $m=r$ and $\rho=1.$
Suppose, further, that there exists a sequence $(M_p)$ satisfying (M.1), (M.2) and (M.3$'$), as well as that $a(\mathbb{R}^n)\cap\Lambda_{l,\zeta,\eta}=\emptyset$ for some constants
$l\geq 1$, $\zeta>0$ and $\eta\in\mathbb{R}.$
In \cite[Example 3.5.30(ii)]{knjigah}, we have proved
that there exist numbers $\eta'\geq\eta ,$ $N\in {\mathbb N}$ and $M\geq 1$ such that
\begin{equation}\label{cale-jj}
\bigl \|R(\lambda:A)f\bigr\|= \Biggl \| {\mathcal F}^{-1} \Biggl( \frac{1}{\lambda -a(\xi)}{\mathcal Ff}(\xi) \Biggr) \Biggr \|\leq M\bigl(1+|\lambda|\bigr)^{N},\quad \lambda \in \Lambda_{l,\zeta,\eta'},\ f\in E.
\end{equation}
Set ${{\mathbb
N_{0}^{l}}}:=\{\eta \in {\mathbb N_{0}^{n}} : \eta_{l+1}=\cdot \cdot
\cdot =\eta_{n}=0\}$ and
$ E_{l}:=\{ f\in E : f^{(\eta )} \in E \mbox{ for all }\eta \in
{{\mathbb N_{0}^{l}}}\}.$ Then the totality of seminorms
$(q_{\eta}(f):=||f^{(\eta )}||_{E},\ f\in E_{l};\ \eta \in {{\mathbb
N_{0}^{l}}})$ induces a Fr\' echet topology on $E_{l}$ (\cite{x263}).
Define now the operator $A_{l}$ on $E_{l}$ by $D(A_{l}):=\{f\in E_{l} : \mathrm{Op}_E(a)f\in E_{l}\}$ and $A_{l}f:=\mathrm{Op}_E(a)f$
($f\in D(A_{l})$). Since
$$
{\mathcal F}^{-1} \Biggl( \frac{1}{\lambda -a(\xi)}\Bigl({\mathcal F}f^{(\eta)}\Bigr)(\xi) \Biggr) =\Biggr({\mathcal F}^{-1} \Biggl( \frac{1}{\lambda -a(\xi)}{\mathcal Ff}(\xi) \Biggr) \Biggr)^{(\eta)},\quad \eta \in {{\mathbb
N_{0}^{l}}},\ f\in E_{l},
$$
it readily follows from (\ref{cale-jj}) that $\Lambda_{l,\zeta,\eta'} \subseteq \rho(A_{l})$ and that for each $\eta \in {{\mathbb
N_{0}^{l}}}$ we have
$$
q_{\eta}\Bigl(R\bigl(\lambda:A_{l}\bigr)f\Bigr) \leq  M\bigl(1+|\lambda|\bigr)^{N}q_{\eta}(f),\quad \lambda \in \Lambda_{l,\zeta,\eta'},\ f\in E_{l}.
$$
Hence, Theorem \ref{tempera-ultra}(ii) and Remark \ref{ultra-logarithmic}(i) taken together imply that $A_{l}$ generates an ultradistribution semigroup of $(M_{p})$-class in $E_{l}.$
\end{example}

We close the paper with the following example.

\begin{example}\label{prusinjo}
Consider the Fr\' echet space $E$ from Example \ref{multiplication}, topologized with the family of seminorms $||f||_{k}=\sum_{j=0}^{k}\sup_{x\geq
0}|f^{(j)}(x)|,$ $f\in E,$ $k\in {{\mathbb N}_{0}}.$
Suppose $c_{0}>0,$ $\beta>0,$ $s>1,$ $l>0,$ $\bar{a}>0$ and define $A$ by $D(A):=\{u\in E : c_{0}u^{\prime}(0)=\beta u(0)\}$ and
$Au:=c_{0}u^{\prime \prime},$ $u\in D(A)$ (\cite{segedin}). Then $A$ cannot be the
generator of a $C_{0}$-semigroup since $D(A)$
is not dense in $E$ (cf. \cite[Definition 1.1, Proposition 1.3]{komura}).
Put $A_{1}:=A/c_{0},$ $M_{p}:=p!^{s},$
$\omega_{l,s}(\lambda):=\prod
_{p=1}^{\infty}(1+\frac{l\lambda}{p^{s}}),$ $\lambda \in {\mathbb
C}$ and $k_{l,s}(t):={\mathcal
L}^{-1}(\frac{1}{\omega_{l,s}(\lambda)})(t),\ t\geq 0.$
If $\lambda \in {\mathbb C} \setminus (-\infty,0],$ then we know that
$\lambda \in \rho (A)$ and
\begin{align*}
& \bigl( \lambda -A \bigr)^{-1}f(x)= \frac{1}{2{\sqrt{c_{0}\lambda}}}\Biggl[ \int ^{x}_{0}e^{-\sqrt{\lambda/c_{0}}(x-s)}f(s)\, ds+\int
^{\infty}_{x}e^{\sqrt{\lambda/c_{0}}(x-s)}f(s)\, ds\Biggr]
\\ &+\frac{1}{c_{0}}\Biggl[
\frac{c_{0}\sqrt{\lambda/c_{0}}-\beta}{c_{0}\sqrt{\lambda/c_{0}}+\beta}\frac{1}{2\sqrt{\lambda/c_{0}}}\int^{\infty}_{0}e^{-\sqrt{\lambda/c_{0}}
s}f(s)\, ds\Biggr]e^{-\sqrt{\lambda/c_{0}} x},\ x\geq 0,\ f\in E.
\end{align*}
Furthermore, for every sufficiently small number $\varepsilon>0$ and for every integer $k\in {\mathbb N}_{0},$ there exist constants $c(\varepsilon,l,s)>0$ and $c(k,\varepsilon,l,s)>0$ such that
\begin{equation}\label{faithdj}
\Bigl \|  (\lambda-A)^{-1}f \Bigr \|_{k}\leq c(k,\varepsilon,l,s)e^{c(\varepsilon,l,s)|\lambda|^{1/s}}\bigl \|f \bigr \|_{k},\quad f\in E,\ \lambda \in \Sigma_{\pi -\varepsilon}.
\end{equation}
Let the numbers $r>0$ and $\theta \in (-\pi/2,\pi/2)$ be fixed. Put
$$
{\mathcal G}_{r,\theta}(\varphi)f:=(-i)\int\limits_{\bar{a}-i\infty}^{\bar{a}
+i\infty}\hat{\varphi}(\lambda)\bigl(\lambda-re^{i\theta}A\bigr)^{-1}f\,d\lambda,\quad f\in E,\ \varphi\in {\mathcal D}^{(M_{p})}.
$$
Then Theorem \ref{tempera-ultra} implies
that ${\mathcal G}_{r,\theta}$ is a pre-(EUDS) of Beurling class, as well as that
$\mathcal{G}_{r,\theta}(\varphi)A\subseteq A{\mathcal G}_{r,\theta}(\varphi),$ $\varphi \in {\mathcal D}^{(M_{p})}$ and that $A\mathcal{G}(\varphi)x=\mathcal{G}\bigl(-\varphi'\bigr)x-\varphi(0)Cx,\quad x\in E,\ \varphi \in {\mathcal D} \ \ (\varphi \in {\mathcal D}^{\ast})$ with $A$ and ${\mathcal G}$ replaced respectively by $re^{i\theta}A$ and ${\mathcal G}_{r,\theta}.$
By \cite[Lemma 1.6.7]{a43}, we have that
\begin{equation}\label{lapla}
{\mathcal L}\Bigl(\bigl(t\pi \bigr)^{(-1)/2}e^{-s^{2}/4t}\Bigr)(\lambda)=e^{-\sqrt{\lambda} s}/\sqrt{\lambda},\quad \Re \lambda>0,\ s>0.
\end{equation}
Keeping in mind this equality, as well as the Fubini theorem and the obvious equality
$$
(-i)\int\limits_{\bar{a}-i\infty}^{\bar{a}
+i\infty}\hat{\varphi}(\lambda)e^{-\lambda t}\, d\lambda=\varphi (t),\quad \varphi\in {\mathcal D}^{(M_{p})},\ t\geq 0,
$$
we get that
$$
(-i)\int\limits_{\bar{a}-i\infty}^{\bar{a}
+i\infty}\hat{\varphi}(\lambda)\frac{e^{-\sqrt{\lambda}s}}{\sqrt{\lambda}}\,d\lambda=\int^{\infty}_{0}\bigl(t\pi \bigr)^{(-1)/2}e^{-s^{2}/4t}\varphi(t) \, dt,\quad
s>0,\
\varphi\in {\mathcal D}^{(M_{p})}.
$$
Using the principle of analytic continuation, it readily follows that
\begin{equation}\label{dropkick}
(-i)\int\limits_{\bar{a}-i\infty}^{\bar{a}
+i\infty}\hat{\varphi}(\lambda)\frac{e^{-\sqrt{\lambda}z}}{\sqrt{\lambda}}\,d\lambda=\int^{\infty}_{0}\bigl(t\pi \bigr)^{(-1)/2}e^{-z^{2}/4t}\varphi(t) \, dt,\quad
z\in \Sigma_{\pi/4},\
\varphi\in {\mathcal D}^{(M_{p})}.
\end{equation}
Taking into account (\ref{dropkick}), the Fubbini theorem, as well as
(\ref{lapla}), it can be simply proved that
\begin{align}
\notag {\mathcal G}_{r,\theta}& (\varphi)f(x)=
\frac{1}{2\sqrt{\pi c_{0}re^{i\theta}}}\int^{\infty}_{0}\varphi(t)\Biggl[ t^{(-1)/2}\int \limits^{x}_{0}e^{-s^{2}e^{i\theta}/4c_{0}tr}f(x-s)\, ds \Biggr]\, dt
\\\notag & +\frac{1}{2\sqrt{\pi c_{0}re^{i\theta}}}\int^{\infty}_{0}\varphi(t)\Biggl[ t^{(-1)/2}\int \limits^{\infty}_{0}e^{-s^{2}e^{i\theta}/4c_{0}tr}f(x+s)\, ds \Biggr]\, dt
\\\label{buda-bar} & +
\frac{1}{2\sqrt{c_{0}re^{i\theta}}} \int^{\infty}_{0}\varphi(t) \Biggl[ \frac{1}{2\pi i}\int\limits_{\bar{a}-i\infty}^{\bar{a}
+i\infty} e^{\lambda t}f(\lambda;x)\, d\lambda \Biggr]\, dt,\ x\geq 0,\ f\in E,\ \varphi\in {\mathcal D}^{(M_{p})}_{0},
\end{align}
where
$$
f(\lambda;x):=\Biggl[ e^{-\sqrt{\lambda r^{-1}e^{-i\theta} c_{0}^{-1}}x}
\frac{c_{0}\sqrt{\lambda  r^{-1}e^{-i\theta} c_{0}^{-1}}-\beta}{c_{0}\sqrt{\lambda  r^{-1}e^{-i\theta} c_{0}^{-1}}+\beta}\frac{1}{\sqrt{\lambda}}\int^{\infty}_{0}e^{-\sqrt{\lambda  r^{-1}e^{-i\theta} c_{0}^{-1}} v}f(v)\, dv \Biggr]
$$
for every $x> 0$ and for every $\lambda \in {\mathbb C}$ with $\Re \lambda =\bar{a}.$ Suppose now that ${\mathcal G}(\varphi)f=0,$ $\varphi \in {\mathcal D}^{(M_{p})}_{0}$ for some $f\in E.$ Then the representation formula (\ref{buda-bar}) implies
\begin{align}
\notag \bigl(t\pi)^{(-1)/2}\int \limits^{x}_{0}e^{-s^{2}e^{i\theta}/4c_{0}tr}f(x-s)\, ds &+ \bigl(t\pi)^{(-1)/2}\int \limits^{\infty}_{0}e^{-s^{2}e^{i\theta}/4c_{0}tr}f(x+s)\, ds
\\\label{murphy}& +\Biggl[ \frac{1}{2\pi i}\int\limits_{\bar{a}-i\infty}^{\bar{a}
+i\infty} e^{\lambda t}f(\lambda;x)\, d\lambda \Biggr]\, dt=0,
\end{align}
for all $x> 0$ and $t>0.$ Define $f_{1}(x):=f(x),$ $x\geq 0$ and $f_{1}(x):=0,$ $x<0.$ Taking the Laplace transform of the both sides of (\ref{murphy}), we obtain with the help of (\ref{lapla}), the  Fubini theorem, as well as the complex inversion formula for the Laplace transform and the principle of analytic continuation,
that
\begin{align}\label{delarno}
\notag \int^{\infty}_{0} \frac{e^{-\sqrt{\lambda c_{0}^{-1}r^{-1}e^{-i\theta}}s}}{\sqrt{\lambda}}f_{1}(x-s)\, ds+\int^{\infty}_{0}\frac{e^{-\sqrt{\lambda c_{0}^{-1}r^{-1}e^{-i\theta}}s}}{\sqrt{\lambda}}f_{1}(x+s)\, ds+f(\lambda ;x)=0,
\end{align}
provided $x>0$ and $\Re \lambda>0.$
Denote $F(t):={\mathcal L}^{-1}(1/(c_{0} \lambda+\beta))(t),$ $t\geq 0.$ Then the uniqueness theorem for the Laplace transform and an elementary argumentation show that
$$
f_{1}(x-s)+f_{1}(x+s)+2\beta \int^{s}_{0}F(s-r)f_{1}(r)\, dr=f_{1}(s-x),\quad
$$
provided that $x,\ s>0$ and $s\neq x.$
Letting $s\rightarrow 0+,$ for fixed $x>0,$ we get $f(x)=0,$ so that ${\mathcal G}_{r,\theta}$ satisfies (C.S.2) with $C=I.$ Hence, we have proved that  ${\mathcal G}_{r,\theta}$ is an (EUDS) generated by $A_{r,\theta}\equiv re^{i\theta}A.$
On the other hand, E. Borel's theorem (\cite{meise}) implies that the operator $A_{r,\theta}$ is not stationary dense. By \cite{statio}, we get that there is no $n\in {\mathbb N}_{0}$ such that $A$ is the generator of a (local) $n$-times integrated semigroup on $E$. This is an improvement of the corresponding result from \cite{segedin}, where we have only proved that the operator $A_{1,0}$ cannot be the generator of an exponentially equicontinuos $n$-integrated semigroup for some $n\in {\mathbb N}_{0}$.
The question whether there exists an injective operator $C\in L(E)$ such that $A$
is the integral generator of an exponentially equicontinuous
$C$-regularized semigroup on $E$ has been proposed in \cite{segedin}. Now we will answer this question in the affirmative. Let $1<s'<s.$ Put
\begin{equation}\label{steve}
S_{r,\theta}(t)f:=\frac{1}{2\pi i}\int \limits_{\bar{a}-i\infty}^{\bar{a}+i\infty}e^{\lambda t}\frac{\bigl(\lambda -A_{r,\theta}\bigr)^{-1}f}{\omega^{(p!^{s'})}(i\lambda)}\, d\lambda,\quad t\geq 0,\ f\in E,
\end{equation}
$C_{r,\theta}:=S_{r,\theta}(0)$ and $C:=C_{1,0}.$
By the proof of Theorem \ref{grace}, it readily follows that $(S_{r,\theta}(t))_{t\geq 0}$ is an exponentially equicontinuous $C_{r,\theta}$-regularized semigroup generated by $A_{r,\theta}$ and that
$C_{r,\theta}(D_{\infty}(A))\subseteq E^{(p!^{s})}(A).$ The straightforward integral computation involving (\ref{faithdj}) shows that
for each $k\in {\mathbb N}_{0}$ there exists a constant $c_{k,r,\theta,s,s'}>0$ such that $\|S_{r,\theta}(t)f\|_{k}\leq c_{k,r,\theta,s,s'}\|f\|_{k},$ $f\in E,$ $t\geq 0,$
so that
$(S_{r,\theta}(t))_{t\geq 0}$ is an equicontinuous $C_{r,\theta}$-regularized semigroup generated by $A_{r,\theta}.$

Denote $q(\lambda)f=(I-A)^{-1}Cf$, for $f\in E$. We need to show that the family $\{{\lambda}q({\lambda})\,:\,\lambda\in\sum_{\frac{\pi}{2}+\gamma}\}$, $\gamma\in(0,\frac{\pi}{2}-|\theta |)$ is equicontinuous and $\lim_{|\lambda|\rightarrow\infty}{\lambda}q({\lambda})f=Cf$, $f\in E$. The first part is a consequence of a simple integral computation involving the resolvent equation and Cauchy theorem, while the second part follows from the representation formula (\ref{steve}). Now, the conditions of \cite[Theorem 3.7]{sic} are fulfilled, so by its application we obtain that $(S_{r,\theta}(t))_{t\geq 0}$ is an equicontinuous analytic $C_{r,\theta}$-regularized semigroup of angle $\frac{\pi}{2}-|\theta |,$ generated by $A_{r,\theta}$.
%

Observe finally that ${\mathcal G}_{r,\theta}$ is an analytic (UDS) of angle $(\pi/2)-|\theta|.$ We can prove with the help of (\ref{buda-bar})
that $G_{r,\theta}(t)\in L(E)$ for all $t>0,$ as well as that
\begin{align*}
G_{r,\theta}(t)f(x)&=\frac{1}{2\sqrt{\pi c_{0}re^{i\theta}}}t^{(-1)/2}\int \limits^{x}_{0}e^{-s^{2}e^{i\theta}/4c_{0}tr}f(x-s)\, ds \, dt
\\& +\frac{1}{2\sqrt{\pi c_{0}re^{i\theta}}} t^{(-1)/2}\int \limits^{\infty}_{0}e^{-s^{2}e^{i\theta}/4c_{0}tr}f(x+s)\, ds
\\& +\frac{1}{2\sqrt{c_{0}re^{i\theta}}} \frac{1}{2\pi i}\int\limits_{\bar{a}-i\infty}^{\bar{a}
+i\infty} e^{\lambda t}f(\lambda;x)\, d\lambda,\quad x\geq 0,\ t>0.
\end{align*}
Now it is clear that the mapping $t\mapsto G_{r,\theta}(t)\in L(E),$ $t>0$ can be analytically extended to the sector $\Sigma_{(\pi/2)-|\theta|},$ as claimed.

Observe, finally, that we can employ the obtained results in the analysis of
the control problem for a one-dimensional heat equation for materials
with memory (cf. \cite[pp. 146-147]{prus})\begin{equation}\label{heqV}
u(t)=A\int^{t}_{0}u(s)\, ds+(a\ast i)(t)f(\cdot),\quad\quad t\geq 0.\end{equation}
Let the function $a(t)$ is Laplace transformable, abs$(a)=0$
and let the analytic function $\hat{a} : {\mathbb C} \setminus
(-\infty,0] \rightarrow {\mathbb C} \setminus (-\infty,0]$ satisfy that
$\hat{a}(\lambda)=\tilde{a}(\lambda),$ $\Re \lambda>0,
$ as well as that for each number $\epsilon\in (0,\pi)$ there exists a number $\delta(\epsilon)\in (0,\pi)$ such that
$\hat{a}(\Sigma_{\pi-\varepsilon}) \subseteq \Sigma_{\pi-\delta(\epsilon)}$
and $\lim_{|\lambda| \rightarrow \infty}|\hat{a}(\lambda)|=0.$
Using again \cite[Theorem 3.7]{sic}, we
obtain that the operator $A$
is the integral generator of an equicontinuous analytic
$(a,C)$-regularized resolvent family $(R(t))_{t\geq 0}$ of
angle $\pi/2$ satisfying additionally that, for every $k\in
{{\mathbb N}_{0}}$ and $\epsilon \in (0,\pi),$ there exists
$c(k,\epsilon)'>0$ with $||R(z)f||_{k}\leq c(k,\epsilon)'||f||_{k},$
$z\in \Sigma_{\pi-\epsilon},$ $f\in E.$ Consider now the problem \cite[(5.68), p. 147]{prus}, with the function $i(t)$ being the intensity of the radiation and with the part $\alpha e^{-\alpha x}$ (which represents the absorption of the radiation in the material) replaced by a function $f(x)$ which belongs to the abstract Beurling space $E^{(p!^{s})}(A)$ for some $s>1.$ Let the function $(a\ast i)(t)$ belongs to $W^{1,1}_{loc}([0,\infty)).$ Since $A$
is the integral generator of an equicontinuous analytic
$(a,C)$-regularized resolvent family $(R(t))_{t\geq 0}$ of
angle $\pi/2$, the conditions of \cite[Theorem 2.6(ii)]{sic} are fulfilled, so (\ref{heqV}) has a unique strong solution which additionally satisfies $u(t)\in D_{\infty}(A),$ $t\geq 0.$
\end{example}

\end{document}